\documentclass[12pt]{amsart}

\usepackage{amsgen, amsmath, amstext, amsbsy, amsopn, amsfonts, amsthm, amssymb}

\newtheorem{thm}{Theorem}
\newtheorem*{prob}{Problem}
\newtheorem{lem}{Lemma}

\theoremstyle{definition}
\newtheorem*{defn}{Definition}

\theoremstyle{remark}

\newcommand{\pt}{\mathrm{pt}}

\DeclareMathOperator{\hind}{hind}

\DeclareMathOperator{\Ext}{Ext}

\DeclareMathOperator{\cat}{cat}

\renewcommand{\int}{\mathop{\rm int}}

\renewcommand{\epsilon}{\varepsilon}

\begin{document}

\title{Periodic billiard trajectories in smooth convex bodies}
\author{R.N. Karasev}
\email{r\_n\_karasev@mail.ru}
\address{
Roman Karasev, Dept. of Mathematics, Moscow Institute of Physics
and Technology, Institutskiy per. 9, Dolgoprudny, Russia 141700}
\thanks{This research was supported by the Russian Foundation for Basic Research grants No. 03-01-00801
and 06-01-00648, and by the President of the Russian Federation
grants No. MK-5724.2006.1 and No. MK-1005.2008.1}


\begin{abstract}
We consider billiard trajectories in a smooth convex body in $\mathbb R^d$ and estimate the number of distinct periodic trajectories that make exactly $p$ reflections per period at the boundary of the body. In the case of prime $p$ we obtain the lower bound $(d-2)(p-1)+2$, which is much better than the previous estimates.
\end{abstract}

\maketitle

\section{Introduction}

First, we give several definitions on billiards in convex bodies.

\begin{defn}
\emph{A billiard trajectory} in a smooth convex body $T\in\mathbb R^d$ is a polygon $P\subset T$, with all its vertices on the boundary of $T$, and at each vertex the direction of line changes according to the elastic reflection rule. 
\end{defn}

\begin{defn}
Let $P$ be a periodic billiard trajectory. Its number of vertices is called its \emph{length}.
\end{defn}

From here on $d$ will denote the dimension of convex body under consideration.
We encode periodic trajectories by the sequence of their vertices $(x_1,x_2,\ldots,x_l)$. The indices are considered $\mod l$. We do not allow coincidences $x_i=x_{i+1}$, but $x_i$ may equal $x_j$ if $i-j\not=0,\pm 1 \mod l$. Of course, the polygon segments can have mutual intersections.

The problem of finding lower bounds on the number of distinct trajectories of a given length has quite a long history. We mean by ``distinct trajectories'' the orbits of the dihedral group, acting naturally on trajectories. 

The first lower bound was obtained in~\cite{bir27} for the two-dimensional case.

In the case of length $2$ and dimension $d$ the paper~\cite{kui64} gives the lower bound $d$. This bound is tight, it may be seen by considering a generic ellipsoid.

The case of dimension $3$ was considered in~\cite{bab90}, but later an error was found, as noted in~\cite{fartab99,far00}.

In the book~\cite{kleewa} the following problem was formulated (Problem 1.7):

\begin{prob}
Let $T\subset\mathbb R^d$ be a smooth convex body. Prove that it has at least $d$ periodic billiard trajectories of length $3$.
\end{prob}

In the papers~\cite{fartab99, far00} some lower bounds were obtained for length $n$. There were two cases: odd $n$ without other assumptions, and odd prime $n$. In the latter case the bounds were better. Here we give another lower bound for prime lengths.

\begin{thm}
\label{bill-p}
Let $T\subset\mathbb R^d$ be a smooth convex body, $d\ge 3$, let $p>2$ be a prime. Then there are at least $(d-2)(p-1)+2$ distinct periodic billiard trajectories of length $p$ in $T$.
\end{thm}

The proof of this theorem is mainly based on results from~\cite{fartab99,far00}, calculating the cohomology ring of relevant configuration spaces. The case of non-prime length is not considered here.

\section{Cohomological index of $Z_p$-spaces}

In the sequel, if we want to consider $Z_p$ as a group of transformations, we denote it $G$. If it is used as a ring of coefficients for cohomology it is denoted $Z_p$.

Consider the cohomology algebra $A_G = H_G^*(\pt, Z_p) = H^*(BG, Z_p)$. It is well known~\cite{hsiang75} that $A_G$ has two generators $v,y$ of dimensions $1$ and $2$ respectively, the relations being
$$
v^2 = 0,\quad\beta(v) = y.
$$
We denote $\beta(x)$ the Bockstein homomorphism. We see that $A_G$ is $Z_p$ in every dimension.

\begin{defn}
Let $X$ be a $G$-space. The maximal $n$ such that the natural map 
$$
H_G^n(\pt, Z_p)\to H_G^n(X, Z_p)
$$
is nontrivial, is called the \emph{cohomological index} of $X$ and denoted $\hind X$.
\end{defn}

Note that if the map $H_G^n(\pt, Z_p)\to H_G^n(X, Z_p)$ is trivial for some $n$, it should be trivial for all larger $n$,
since every element of $A_G$ can be obtained from any nonzero element of less dimension by multiplications by $v$,$y$, elements of $Z_p^*$, and by applying the Bockstein homomorphism. 

Every $G$-space $X$ gives a fibration
$$
(X\times EG)/G\to BG.
$$

If $G$ acts on cohomology of $X$ trivially, there exists a spectral sequence with $E_2$-term equal to $H^*(X, Z_p)\otimes A_G$, converging to $H_G^*(X, Z_p)$. If $G$ has a nontrivial action on cohomology, $E_2$-term is
$E_2^{x,y} = H^x(G, H^y(X, Z_p))$, where $H^*(G, M) = \Ext_{Z_p[G]}^*(M, Z_p)$ is the cohomology of a $G$-module $M$.

If $G$ acts freely on $X$, we have $H_G^*(X, Z_p) = H^*(X/G, Z_p)$. In this case $\hind X$ does not exceed the dimension of $X$. 

The following property is obvious by the definition of index:

\begin{lem}[Monotonicity of index]
If there is a $G$-equivariant map $X\to Y$, then
$$
\hind X\le \hind Y.
$$
\end{lem}

\section{Lyusternik-Schnirelmann theory}

We recall the definition of the Lyusternik-Schnirelmann category of a topological space.

\begin{defn}
\emph{The relative category} of a pair $Y\subseteq X$ is the minimal cardinality of a family of contractible in $X$ open subsets of $X$, covering $Y$. It is denoted $\cat_X Y$. The number $\cat_X X = \cat X$ is called \emph{the category} of $X$.
\end{defn}

Here we state the Lyusternik-Schnirelmann theorem on the number of critical points in the form it is used in~\cite{fartab99} for manifolds with boundary.

\begin{thm}
\label{ls-crit}
Let $X$ be a compact smooth manifold with boundary. Let $f :X\to\mathbb R$ be a smooth function, its gradient at $\partial X$ always having strictly outward (inward) direction. Then $f$ has at least $\cat X$ critical points.
\end{thm} 

We will need the following lemma, that is a special case of the results of~\cite{fh92}.

\begin{lem}
\label{ls-homlenboc}
Let a space $X$ have $k$ cohomology classes $x_1,\ldots,x_k\in H^1(X, Z_p)$ and $l$ classes of arbitrary positive dimension $y_1,\ldots,y_l\in H^*(X, Z_p)$. If
$$
\beta(x_1)\beta(x_2)\dots\beta(x_k)y_1\dots y_l \not= 0,
$$
then $\cat X\ge 2k + l + 1$.
\end{lem}

Actually, we use the following corollary of the previous lemma.

\begin{lem}
\label{ls-ind}
Let a finite group $G$ act freely on a space $X$, assume that for some subgroup of prime order $G'\subseteq G$ we have $\hind_{G'} X\ge n$. Then $\cat X/G\ge n+1$.
\end{lem}

\begin{proof}
As it was noted in~\cite{far00} $\cat X/G \ge \cat X/G'$. 

$G'$ acts freely on $X$, hence $H^*(X/G', Z_p) = H_{G'}^*(X, Z_p)$. The latter algebra has a nonzero product of $A_{G'}$ elements $y^{\frac{n}{2}}$ or $vy^{\frac{n-1}{2}}$, for even or odd $n$ respectively. So we can apply Lemma~\ref{ls-homlenboc}.
\end{proof}

\section{The configuration space}

Following the papers~\cite{fartab99, far00} we describe the configuration spaces that arise naturally in the billiard problems. For a topological space $X$ denote
$$
G(X, p) = \{(x_1,\ldots,x_p)\in X^{p} : x_1\not=x_2,x_2\not=x_3,\ldots,x_{p-1}\not=x_p,x_p\not=x_1\}.
$$

The space $G(X, p)$ has an action of the dihedral group $D_p$, which has a subgroup of even permutations, isomorphic to $Z_p=G$.

Proposition 4.1 of~\cite{fartab99} claims that $G(\partial T, p)$ contains a $D_p$-invariant compact manifold with boundary $G_\varepsilon(\partial T, p)$, which is $D_p$-equivariantly homotopy equivalent to $G(\partial T, p)$.

In this section we find the cohomological $G$-index of $G(\partial T, p)$, obviously equal to the index of $G_\varepsilon(\partial T, p)$. The space $\partial T$ is homeomorphic to a $d-1$-sphere, so we write $G(S^{d-1}, p)$.

The first lower bound on the index of $G(S^{d-1}, p)$ will be obtained by considering another space: $G(\mathbb R^d, p)$. Recall a special case of Proposition 2.2 from~\cite{fartab99}.

\begin{thm}
\label{Rd-conf-space-hom}
Let $d\ge 2$. The cohomology algebra $H^*(G(\mathbb R^d, p), Z_p)$ is multiplicatively generated by $d-1$-dimensional classes $s_1,\ldots, s_p$ and relations
$$
s_1^2 = s_2^2 = \dots = s_p^2 = 0,\quad s_1s_2\dots s_{p-1} + s_2s_3\dots s_p + s_3s_4\dots s_ps_1
+ s_ps_1\dots s_{p-2} = 0.
$$
The group $G$ acts on $s_1,\ldots, s_p$ by cyclic permutations.
\end{thm}

Considering the action of $G$ on $G(\mathbb R^d, p)$, we deduce a corollary.

\begin{thm}
\label{Rd-conf-space-ind}
If $d\ge 2$, then $\hind G(\mathbb R^d, p) = (d-1)(p-1)$.
\end{thm}

\begin{proof}
Denote $Y = G(\mathbb R^d, p)$.

Consider a spectral sequence with $E_2^{x,y} = H^x(G, H^y(Y, Z_p))$. Note that $H^*(Y, Z_p)$ are free $Z_p[G]$-modules in dimensions between $0$ and $(d-1)(p-1)$. In dimension $0$ it is $Z_p$, in dimension $(d-1)(p-1)$ it is $Z_p[G]/Z_p$.

Let $I_G^k$ be the ideal of $A_G$, consisting of elements of dimension at least $k$.

In $E_2$-term the bottom row is $A_G$, the top row is $I_G^1$ with shifted by $-1$ grading. The latter claim is deduced from the cohomology exact sequence for $0\to Z_p\to Z_p[G]\to Z_p[G]/Z_p\to 0$.

All other rows of $E_2$ have nontrivial groups in $0$-th column only. Every differential of the spectral sequence is a homomorphism of $A_G$-modules. Hence, the intermediate rows cannot be mapped non-trivially to the bottom row. The only nonzero differential can map the top row to a part of the bottom row, isomorphic to $I_G^{(d-1)(p-1) + 1}$. This map has to be nontrivial because the index of $Y$ is finite.
\end{proof}

We see that if $X$ contains $\mathbb R^d$, then $G(X, p)\supseteq G(\mathbb R^d, p)$. By the monotonicity of index we have $\hind G(X, p) \ge (d-1)(p-1)$, in particular $G(S^{d-1}, p)\ge (d-2)(p-1)$.

Then we need more precise estimates on $\hind G(S^{d-1},p)$. Recall two theorems from~\cite{far00} (Theorem 18, Theorem 19), describing the algebra $H^*(G(S^{d-1},p), Z_p)$. We use the notation $[n] = \{1,2,\ldots, n\}$.

\begin{thm}
\label{conf-space-hom-even}
Let $d\ge 4$ be even. Then $H^*(G(S^{d-1},p), Z_p)$ is generated by the elements
$$
u\in H^{d-1}(G(S^{d-1},p), Z_p),\quad s_i\in H^{i(d-2)}(G(S^{d-1},p), Z_p), i\in [p-2]
$$
and relations
$$
u^2 = 0,\quad s_is_j = \frac{(i+j)!}{i!j!}s_{i+j},\ \text{if}\ i+j\le p-2,\ \text{or otherwise}\ s_is_j=0.
$$
\end{thm}

\begin{thm}
\label{conf-space-hom-odd}
Let $d\ge 3$ be odd. Then $H^*(G(S^{d-1},p), Z_p)$ is generated by the elements
$$
w\in H^{2d-3}(G(S^{d-1},p), Z_p),\quad t_i\in H^{i(2d-4)}(G(S^{d-1},p), Z_p), i\in \left[\frac{p-3}{2}\right]
$$
and relations
$$
w^2 = 0,\quad t_it_j = \frac{(i+j)!}{i!j!}t_{i+j},\ \text{if}\ i+j\le \frac{p-3}{2},\ \text{or otherwise}\ t_it_j=0.
$$
\end{thm}

Theorem~\ref{conf-space-hom-even} is formulated in~\cite{far00} for the cohomology coefficients $\mathbb Q$. But the proof is valid without change for $Z_p$, since we only have to divide by $i!$, where $i\le p-2$.

Note that $G$ acts on the cohomology in these theorems trivially, since the cohomology is either $Z_p$ or $0$ in every dimension.

Finally, let us give another lower bound for $\hind G(S^{d-1}, p)$.
We have a $D_p$-equivariant map from the Stiefel variety of $2$-frames in $\mathbb R^d$ $V_d^2\to G(S^{d-1}, p)$, defined as follows. Take some regular $p$-gon in the plane. Every frame $(e_1, e_2)\in V_d^2$ gives an embedding of this $p$-gon to $S^{d-1}$, which is $D_p$-equivariant. So $\hind G(S^{d-1},p)\ge \hind V_d^2$, the latter index being equal to (see~\cite{mak89}) $2d-3$.

\begin{thm}
\label{conf-space-ind} If $d\ge 3$, then $\hind G(S^{d-1},p) = (d-2)(p-1) + 1$.
\end{thm}

\begin{proof}
First, as it was mentioned in~\cite{bab90,fartab99}, the Morse theory shows that $G(S^{d-1}, p)$ is homotopy equivalent to a CW-complex of dimension $(d-2)(p-1)+1$. So the index is naturally bounded from above and we prove the lower bound only.

If $p=3$ $(p-1)(d-2)+1 = 2d - 3$ and the lower bound is already done. So we consider the case $p\ge 5$.

Note that $H^*(G(S^{d-1},p), Z_p)$ is multiplicatively generated by $(u, s_1)$ and $(w, t_1)$ for even and odd $d$ respectively.

First consider the case of even $d$. 

Consider the spectral sequence with $E_2=H^*(G(S^{d-1},p), Z_p)\otimes A_G$. In this case the multiplicative generators of $H^*(G(S^{d-1},p), Z_p)$ have dimension at most $d-1$. If the differentials $d_m$ were trivial for $m\le d$, then they should be trivial for $m>d$ from the dimension of generators. But in this case $\hind G(S^{d-1},p)=+\infty$, which is false. So some of $d_m$ is nontrivial for $m\le d$.

Note that in $E_2$ the bottom row is $A_G$, the next from bottom is $s_1A_G$, then $uA_G$. Let $d_m$ be the first nontrivial differential. The generators $u$ and $s_1$ cannot be mapped non-trivially to the bottom row, since in this case $\hind G(S^{d-1},p)$ would be at most $d-1$. The only possibility for nontrivial $d_m$ is $d_2$ that has $d_2(u) = as_1y$ ($a\in Z_p^*$). In this case $E_3$ will be multiplicatively generated by $v, y\in A_G$, the image of $s_1$ and of $us^{p-2}$, denote the latter by $z$. The relations would be (besides those of $A_G$):
$$
s_1^{p-1} = 0,\quad s_1y = 0,\quad z^2 = 0.
$$
That means that $E_3$ has two full rows in the top and in the bottom, and several rows in between contain $s_1^k$ and $s_1^kv$ for $k\in [p-2]$.

The next $d_m$ cannot map $z$ non-trivially if $m\le (p-1)(d-2) + 1 = \dim z$. They also cannot map $s_1$ non-trivially from the dimension considerations and the lower bound for index. The latter statement can also be deduced from the fact that all differentials are homomorphisms of $A_G$-modules. So finally $d_{(p-1)(d-2)+2}$ has to map $z$ to the bottom row and we have $\hind G(S^{d-1},p) = (p-1)(d-2)+1$.

Now consider the case of odd $d$. Everything is quite the same, but besides the alternative that gives $\hind G(S^{d-1},p) = (p-1)(d-2)+1$ we have another alternative. It could happen that $d_2$ is trivial and the first nontrivial $d_{2d-2}$ maps $w$ non-trivially to the bottom row. But in this case $\hind G(S^{d-1},p) = 2d-3$, which is less then the bound $(d-2)(p-1)$ for $p\ge 5$.
\end{proof}

\section{Proof of Theorem~\ref{bill-p}}

Consider the function on $G(\partial T, p)$
$$
f: (x_1,\ldots, x_p) \mapsto \sum_{i=1}^p |x_ix_{i+1}|.
$$
If $\dfrac{\partial f}{\partial x_i} = 0$, then at $x_i$ the polyline $x_1x_2\ldots x_px_1$ reflects by the elastic reflection rule. $f$ can be considered as a function on $G(\partial T, p)/D_p$, its critical points being in one-to-one correspondence with distinct periodic billiard trajectories of length $p$. 

If $T$ is not strictly convex, some line segments of the trajectory can lie on $\partial T$. But in this case all the segments must lie on the same line, which is impossible.

Following~\cite{fartab99} we consider $G_\varepsilon(\partial T, p)$ instead of $G(\partial T, p)$ and apply Theorem~\ref{ls-crit} to $f$.

Then we estimate $\cat G(\partial T, p)/D_p\ge (d-2)(p-1)+2$ by Theorem~\ref{conf-space-ind} and Lemma~\ref{ls-ind}.

\section{Conclusion}

For an arbitrary $d$-manifold $M$ ($d\ge 2$) $\hind G(M, p)$ is between $(d-1)(p-1)$ (Theorem~\ref{Rd-conf-space-ind}) and $(d-1)(p-1)+1$ (by the dimension considerations). It seems probable that for closed $M$ $\hind G(M, p) = (d-1)(p-1)+1$, as it is for spheres (Theorem~\ref{conf-space-ind}).

The author would like to thank A. Yu. Volovikov for useful discussions on the subject of the paper.

\end{document}